\documentclass[12pt,a4paper,draft,reqno]{amsart}

\usepackage{amssymb}
\usepackage{amsthm}
\usepackage{color}
\usepackage{amsfonts,amssymb,amsmath}

\def\C{\mathbb{C}}
\def\R{\mathbb{R}}
\def\T{\mathbb{T}}
\def\D{\mathbb{D}}
\def\N{\mathbb{N}}

\def\Rat{\mathcal{R}}
\def\RatU{\mathcal{RU}}
\def\PolU{\mathcal{PU}}
\def\Alg{\mathcal{A}}
\def\ND{\mathit{ND}}
\def\Area{\mathop{\mathrm{Area}}}
\def\Len{\mathop{\mathrm{Len}}}
\def\sHol{\mathop{\textsl{Hol}}}
\def\re{\mathop{\textrm{Re}}}

\let\vGa\varGamma
\let\Om\Omega
\let\ze\zeta
\let\ep\varepsilon

\let\de\delta
\let\ga\gamma
\let\te\theta
\let\be\beta
\let\sig\sigma
\let\al\alpha

\let\geq\geqslant
\let\leq\leqslant
\let\eset\varnothing
\let\ov\overline
\let\d\partial
\def\noteq{\not\equiv}
\def\eqref#1{(\ref{#1})}

\newtheorem{thm}{\bf Theorem}[section]
\newtheorem{pro}[thm]{\bf Proposition}
\newtheorem{lem}[thm]{\bf Lemma}
\newtheorem{prb}[thm]{\bf Problem}
\newtheorem{dfn}[thm]{\bf Definition}

\theoremstyle{remark}
\newtheorem{rem}[thm]{\bf Remark}

\numberwithin{equation}{section}

\begin{document}
\sloppy

\title[$L^1$-estimates of derivatives of univalent rational functions]{On
$L^1$-estimates of derivatives \\ of univalent rational functions}

\author{Anton D. Baranov, Konstantin Yu. Fedorovskiy}

\begin{abstract}
We study the growth of the quantity
$$
\int_{\T}|R'(z)|\,dm(z)
$$
for rational functions $R$ of degree $n$, which are bounded and
univalent in the unit disk, and prove that this quantity may grow as
$n^\ga$, $\ga>0$, when $n\to\infty$. Some applications of this result
to problems of regularity of boundaries of Nevanlinna domains are
considered. We also discuss a related result by Dolzhenko which applies
to general (non-univalent) rational functions.
\end{abstract}



\address{
\hskip -1em Anton D. Baranov:
\newline
Department of Mathematics and Mechanics, St.~Petersburg State
University, St.~Petersburg, Russia;
\newline
National Research University Higher School of Economics,
St.~Petersburg, Russia
\newline {\tt anton.d.baranov@gmail.com}
\newline
\newline
Konstantin Yu. Fedorovskiy:
\newline
Bauman Moscow State Technical University,
Moscow, Russia;
\newline
Department of Mathematics and Mechanics, St.~Petersburg State
University, St.~Petersburg, Russia
\newline {\tt kfedorovs@yandex.ru}
\newline\newline \phantom{x}
}

\thanks{The first author was partially supported by RFBR (project~12-01-00434)
and by Dmitry Zimin's `Dynasty' Foundation. The second author was
partially supported by RFBR (projects~12-01-00434 and~15-01-07531), by 
Russian Ministry of Education and Science (project no.~2640), by Dmitry 
Zimin's `Dynasty' Foundation, and by Simons-IUM fellowship.}

\maketitle

\section{Introduction}

For $n\in\N$, denote by $\Rat_n$ the set of all rational functions in
the complex variable of degree at most $n$. Thus, $\Rat_n$ consists of
all functions of the form $P(z)/Q(z)$, where $P$ and $Q$ are
polynomials in the complex variable $z$ such that $\deg{P}\leq n$ and
$\deg{Q}\leq n$.

Let $\D=\{z\in\C\colon |z|<1\}$ be the unit disk in the complex plane
$\C$ and let $\T=\{z\in\C\colon |z|=1\}$ be the unit circle. We denote
by $\RatU_n$ the set of all functions from the class $\Rat_n$ which
have no poles in the closed disk $\ov\D$ and are univalent in $\D$.

In this paper we study the quantity
\begin{equation}\label{h1}
\ell(R):=\int_{\T}|R'(z)|\,dm(z)
\end{equation}
for $R\in\RatU_n$, where $m$ stands for the normalized Lebesgue measure
on $\T$, that is, $dm(z)=|dz|/(2\pi)$. The quantity $\ell(R)$ is the
length of the boundary of the domain $R(\D)$ under the mapping by the
rational univalent function $R$. Our aim is to determine how $\ell(R)$
may grow when $n\to\infty$. More precisely, we are interested in the
following problem:

\begin{prb}\label{pr1}
To determine \textup(or estimate\textup) the value of the quantity
$$
\ga_0:=\limsup_{n\to\infty}\sup_{\substack{R\in\RatU_n,\\[0.25ex]
\|R\|_{\infty,\T}\leq1}}\frac{\log\ell(R)}{\log{n}},
$$
where for $Y\subset\C$ we set $\|f\|_{\infty,Y}=\sup\{|f(z)|\colon z\in
Y\}$.
\end{prb}

It follows from our main result (see Theorem~\ref{t0} below), that
$0<\ga_0\leq\frac12$. Let us briefly explain how and where this problem
(which is of independent interest) has arisen. The concept of a
Nevanlinna domain appeared recently and quite naturally in problems of
uniform approximation of functions by polyanalytic polynomials on
compact subsets of the complex plane. We formally define and discuss
this class of domains and the corresponding approximation problem (with
all necessary references) in the last section of the paper. Now we
mention only that an intriguing open problem about properties of
Nevanlinna domains (posed in \cite{fed01}) is the question of whether
Nevanlinna domains with unrectifiable boundaries exist. Since any
function from the class $\RatU_n$ maps $\D$ conformally onto some
Nevanlinna domain, the fact that $\ga_0>0$ supports the conjecture that
Nevanlinna domains with unrectifiable boundaries do exist (see
Section~4 for details).

\medskip
Let us state now the main result of this paper. The estimate of $\ga_0$
from below is formulated in terms of the number $B_b(1)$, where
$B_b(t)$, $t\in\R$, is the integral means spectrum for bounded
univalent functions (see Section~3 below). It is worth to mention that,
by the known estimates, $0.23<B_b(1)\leq 0.46$.

\begin{thm}\label{t0}
The following inequalities hold:
\begin{equation}\label{e-main}
B_b(1)\leq\ga_0\leq\frac12.
\end{equation}
\end{thm}

One should compare Theorem~\ref{t0} with a recent result by I.~Kayumov
\cite{kay08} who considered the quantity
$$
\vGa(t):=\limsup_{n\to\infty}\sup_{P\in\PolU_n}
\frac{1}{\log{n}}\log\int_0^{2\pi}|P'(e^{i\theta})|^td\theta,
$$
where $\PolU_n$ denotes the set of all polynomials of degree at most
$n$ which belong to the class $S$ of all univalent functions $f$ in the
unit disk $\D$ normalized by conditions $f(0)=0$, $f'(0)=1$. It is
proved in \cite{kay08} that $\vGa(t)$ coincides with $B(t)$, the
integral means spectrum for the class $S$, for any $t\in\R$ (let us
emphasize that $B(t)$ differs from the integral means spectrum $B_b(t)$
for bounded univalent functions).

We give two proofs for the lower bound in Theorem~\ref{t0}. The first
is based on the standard Runge approximation scheme, while the second
uses an idea from \cite{kay08}. However, we do not know whether it is
true that $\ga_0=B_b(1)$. The problem of finding the value of $\ga_0$
remains open.

The estimate of $\ga_0$ from above in Theorem~\ref{t0} is the
consequence of the following fact.

\begin{pro}\label{pro-upp}
Let $R\in\RatU_n$ and $\|R\|_{\infty,\T}\leq1$. Then
\begin{equation}\label{e-upp}
\int_{\T}|R'(z)|\,dm(z)\leq 6\pi\sqrt{n}.
\end{equation}
\end{pro}

The estimate \eqref{e-upp} may be obtained as a consequence of a deep
inversion of the Hardy--Littlewood embedding theorem for rational
functions which is due to E.M. Dyn'kin \cite{dyn00}. We also give a
direct short proof of \eqref{e-upp}. Our main result is, however, the
estimate from below in \eqref{e-main}. It implies that the length
$\ell(R)$ of the boundary of the domain $R(\D)$ under the mapping by a
function $R\in\RatU_n$ may grow at least as $n^\ga$ (with
$0<\ga<B_b(1)$) when $n\to\infty$.

It is worth to compare the result of Proposition~\ref{pro-upp} with the
known estimates for rational functions which do not satisfy the
univalence condition. It follows from a theorem due to E.\,P.~Dolzhenko
\cite{dol78} (see the next section for details) that any function
$R\in\Rat_n$ which has no poles on $\T$ admits the estimate
$\int_\T|R'(z)|\,dm(z)\leq n\|R\|_{\infty,\T}$ which is sharp (the
equality is attained, e.g., on the function $R(z)=z^n$ as well as on
any finite Blaschke product of degree $n$).

Unfortunately this very interesting and important result remained
unnoticed till recently whereas several of its particular cases and
corollaries were rediscovered by other authors. Most likely this
happened because the paper \cite{dol78} was published in Russian
(though in an international journal) and has never been translated into
English. Thus, an extra aim of this paper is to set the record straight
and present a sketch of the proof of Dolzhenko's theorem which is much
shorter than the original one.

The structure of the paper is as follows. Section~2 is devoted to
integral estimates of derivatives for general rational functions. Here
we discuss Dolzhenko's theorem and give a sketch of a proof.
Furthermore, we formulate and prove Proposition~\ref{p1} which is the
main ingredient of the proof of the first statement in
Theorem~\ref{t0}. Section~3 consists of five subsections. In the first
one we discuss some simple criteria which imply that a given rational
function belongs to the class $\RatU_n$. The upper bound from
Theorem~\ref{t0} is proved in Subsection~3.2, while in the last three
subsections two different proofs of the lower bound will be given.
Finally, in Section~4 we discuss the concept of a Nevanlinna domain and
the problem about uniform approximation by polyanalytic polynomials. We
also give an interpretation of Theorem~\ref{t0} in terms of properties
of boundaries of quadrature domains.

\bigskip

\section{Dolzhenko's theorem, Spijker's lemma and other integral estimates
of derivatives of functions from the class $\Rat_n$}

For integers $n$ and $k$ let $\Alg_{n,k}$ be the set of all algebraic
functions of order $(n,k)$ which means that each function
$f\in\Alg_{n,k}$ is an analytic (multivalued) function in $\ov\C$
except at most finite number of points and satisfies the equation
$P(z,f(z))=0$, where $P(z,w)$ is some polynomial in two complex
variables $z$ and $w$ such that $\deg_zP\leq n$ and $\deg_wP\leq k$.
Thus, any rational function of degree at most $n$ is algebraic function
of order $(n,1)$, so that $\Rat_n=\Alg_{n,1}$.

For a subset $G$ of $\C$ let $\Alg_{n,k}(G)$ denote the class of all
functions $f$ which are defined on $G$ and satisfy the equation
$P(z,f(z))=0$ for every $z\in G$ (where $P$ is as above).

In 1978 E.\,P.~Dolzhenko \cite{dol78} obtained the following remarkable
result about the behavior of the integral \eqref{h1} for functions
from the class $\Alg_{n,k}$:

\begin{thm}[\bf Dolzhenko's theorem]
Let $G$ be an open subset of the circle or of the line $\vGa\subset\C$ and
let a function $f\in\Alg_{n,k}(G)$ be continuous on $G$. Then for any
measurable \textup(with respect to the Lebesgue measure $m_\vGa$ on
$\vGa$\textup) subset $E\subset G$ one has
\begin{equation}\label{est-dol}
\int_E|f'(z)|\,dm_\vGa(z)\leq 2\pi{}nk\|f\|_{\infty,E}.
\end{equation}
In particular, for any $R\in\Rat_n$ and for any measurable subset
$E\subset\T$ of positive measure the estimate
\begin{equation}\label{est-dol-rat}
\int_E|R'(z)|\,dm(z)\leq n\|R\|_{\infty,E}
\end{equation}
holds and is sharp.
\end{thm}

\begin{proof}[Sketch of the proof of inequality \eqref{est-dol}.]
The original proof of Dolzhenko's theorem is long and technically
involved and, thus, its exposition in corpore is beyond the scope of
this paper. Nevertheless, we briefly explain how inequality
\eqref{est-dol} can be obtained by a simpler argument.

Without loss of generality we assume that $\vGa=\R$ and that $E$
consists of finite number of non-overlapping segments (or intervals).
Denoting by $\Len(f(E))$ the length of the full $f$-image of $E$
(counting multiplicities) and applying the classical Crofton's formula
(see, e.g., \cite[Theorem~8]{tab}) we obtain that
\begin{equation}
\label{eq-xx}
\int_E|f'(z)|\,dm_\vGa(z)=\Len(f(E))\leq\frac14\int\#\big(L_{\te, b}\cap
f(E)\big)\,dM_L,
\end{equation}
where $L_{\te, b}$ denotes the line defined by the equation
$x\cos\te+y\sin\te+b=0$, $0\leq\te<2\pi$, $b\in\R$, and the integral is
taken over the measure $dM_L=dbd\te$ on the set of all oriented lines
in the plane. By $\#Y$ we denote the number of elements in the set $Y$.

Furthermore, it is not difficult to check that the curve $f(\vGa)$
intersects each line in the plane at most in $2nk$ points. Also note
that $f(E)\subset D(0,\|f\|_{\infty,E}):=\{z\colon
|z|\leq\|f\|_{\infty,E}\}$ and
$L_{\te,b}\cap D(0,\|f\|_{\infty,E})=\eset$ for $|b|\geq
\|f\|_{\infty,E}$. Hence, the last integral in \eqref{eq-xx} admits the
following easy estimate
$$
\int\#\big(L_{\te, b}\cap f(E)\big)\,dM_L\leq
2nk \cdot 2\pi \cdot 2\|f\|_{\infty,E} =  8\pi nk \|f\|_{\infty,E},
$$
which immediately gives \eqref{est-dol}.
\end{proof}

\begin{rem}
The original proof of Dolzhenko's theorem in \cite{dol78} is also based
on application of special formulas from integral geometry. It is
interesting to note that Crofton's formula not only implies inequality
\eqref{est-dol}, but also gives a sharp constant in it.
\end{rem}

One important particular case of the inequality \eqref{est-dol-rat} was
rediscovered in 1991 by M.\,N.~Spijker \cite{spi91}, who proved that
$$
\int_\T|R'(z)|\,dm(z)\leq n\|R\|_{\infty,\T}
$$
for any $R\in\Rat_n$ without poles on $\T$. The latter result turned
out to be very useful in the context of applied matrix theory (for a
detailed account on this topic see the survey by N.\,Nikolski
\cite{nik13}). Now this result is known as Spijker's lemma, but in all
fairness it should be referred to as the Dolzhenko--Spijker lemma in order
to underline the crucial contribution of Dolzhenko to the themes under
consideration.

\medskip
Now we give another integral estimate of derivatives of functions from
the class $\Rat_n$, which will be used in the proof of
Theorem~\ref{t0}. Let us denote by $m_2$ the normalized planar Lebesgue
measure in $\D$, that is, $dm_2(w)=dudv/\pi$, $w=u+iv$.

\begin{pro}\label{p1}
Let $R\in\Rat_n$ have no poles in $\overline{\D}$. Then
\begin{equation}
\label{est-p1} \int_\T|R'(z)|\,dm(z)\leq
6\sqrt{n}\left(\int_{\D}|R'(w)|^2\,dm_2(w)\right)^{1/2}.
\end{equation}
\end{pro}

Proposition \ref{p1} is a special case of a reverse Hardy--Littlewood
embedding theorem by E.~Dyn'kin \cite{dyn00}. Let $H^\sigma$, $\sig>0$,
denote the standard Hardy space in the disk and let $A^p_\al$, $1\leq
p<\infty$, $\al>-1$, be the weighted Bergman space of functions
analytic in $\D$ with finite norm
$$
\|f\|^p_{A^p_\al}=\int_{\D}|f(z)|^p(1-|z|)^\al dm_2(z).
$$
Let $\sig=\frac{p}{2+\al}$. Dyn'kin \cite[Theorem 4.1]{dyn00} has shown
that for $R\in\Rat_n$ without poles in $\ov\D$ one has
\begin{equation}\label{est-dy}
\|R\|_{H^\sig}\leq Cn^{\frac{1+\al}{p}}\|R\|_{A^p_\al}
\end{equation}
with some absolute constant $C>0$. If we put $p=2$, $\al=0$ (thus,
$\sig=1$) and apply inequality \eqref{est-dy} to $R'$ in place of $R$,
we obtain \eqref{est-p1} with some constant. However, the proof of the
general result of Dyn'kin is rather complicated (it involves, e.g., the
Carleson corona construction). Therefore, we prefer to give a simple
direct proof of Proposition~\ref{p1}.

\begin{proof}[Proof of Proposition~\ref{p1}.]
Let $z_1,\dots,z_n$ be such points in $\D$ that
$1/\ov{z}_1,\dots,1/\ov{z}_n$ are all poles of the function $R$
counting with multiplicities (it is clear that without loss of
generality we may assume that $R$ has exactly $n$ poles). Let
$$
B(z)=\prod_{k=1}^n\frac{\ov{z}_k}{|z_k|}\cdot\frac{z-z_k}{\ov{z}_kz-1}
$$
be the (finite) Blaschke product with zeros at $z_1,\dots,z_n$.
Therefore, there exists $a\in\C$ such that the function $g=R-a$ belongs
to the finite-dimensional model space $K_B=H^2\ominus BH^2$ (for a
systematic exposition of the theory of model spaces see, for instance,
N.\,K.~Nikolski's book \cite{nik80b}). Recall that the function
$$
k_w(z)=\frac{1-\ov{B(w)}B(z)}{1-\ov{w}z}
$$
is the reproducing kernel for the space $K_B$, which means that $k_w\in
K_B$ and $(f,k_w)=2\pi{}if(w)$ for any $f\in K_B$ (the brackets denote
the inner product in $H^2$). It follows from the standard Cauchy
formula and from the fact that the function
$$
\frac1{(1-\ov{w}z)^2}-\biggl(\frac{1-\ov{B(w)}B(z)}{1-\ov{w}z}\biggr)^2
$$
(as a function in the variable $z$) belongs to the space
$BH^2$ that the following integral representation takes place for
every $z\in\D$:
$$
R'(z)=\int_{\T}R(w)\ov{w}\biggl(\frac{1-\ov{B(w)}B(z)}{1-\ov{w}z}\biggr)^2\,dm(w).
$$
Since $B$ is a finite Blaschke product (thus, its zeros have no
accumulation points on $\T$), the latter representation also takes
place for every $z\in\ov\D$.

By a simple version of Green's formula,
$$
\int_{\T}R(w)\ov{w}\biggl(\frac{1-\ov{B(w)}B(z)}{1-\ov{w}z}\biggr)^2\,dm(w)=
\int_{\D}R'(w)\biggl(\frac{1-\ov{B(w)}B(z)}{1-\ov{w}z}\biggr)^2\,dm_2(w).
$$
Therefore, applying the Fubini theorem and the H\"older inequality, we have
$$
\begin{aligned}
\int_{\T}|R'(z)|\,dm(z) & \leq \int_{\D}|R'(w)|\int_{\T}
\biggl|\frac{1-\ov{B(w)}B(z)}{1-\ov{w}z}\biggr|^2\,dm(z)\,dm_2(w)= \\
& = \int_{\D}|R'(w)|\frac{1-|B(w)|^2}{1-|w|^2}\,dm_2(w)\leq \\
& \leq \left(\int_{\D}|R'(w)|^2\,dm_2(w)\right)^{1/2}
\left(\int_{\D}\biggl(\frac{1-|B(w)|^2}{1-|w|^2}\biggr)^2\,dm_2(w)\right)^{1/2}.
\end{aligned}
$$

It is clear, that
$$
\int_{\D}\biggl(\frac{1-|B(w)|^2}{1-|w|^2}\biggr)^2\,dm_2(w) \leq
4\int_{\D}\biggl(\frac{1-|B(w)|}{1-|w|}\biggr)^2\,dm_2(w).
$$
To estimate the last integral we use the following lemma
(see Theorem~3.2 in \cite{dyn00}):

\begin{lem}\label{ldyn}
For $r\in(0,1]$ let
$$
L(r)=\int_{\{z\colon |z|<r\}}
\biggl(\frac{1-|B(w)|}{1-|w|}\biggr)^2\,dm_2(w).
$$
Then
\begin{equation}\label{est-dyn}
L(1)\leq8n+1.
\end{equation}
\end{lem}

As it was mentioned above, the proof of this lemma may be found in
\cite{dyn00}. But we include it here for the sake of completeness and
for the reader's convenience. First of all, it is easy to check that
$L(1/2)\leq1$. Then integrating in polar coordinates $w=re^{it}$ and
using the classical Hardy inequality, we have
$$
\begin{aligned}
\frac1\pi\int_{1/2}^1\frac{r\,dr}{(1-r)^2}
& \int_0^{2\pi}\bigl(1-|B(re^{it})|\bigr)^2\,dt \leq  \\
& \leq \frac1\pi\int_{1/2}^1\frac{r\,dr}{(1-r)^2}
\int_0^{2\pi}\bigl|B(e^{it})-B(re^{it})\bigr|^2\,dt \leq \\
& \leq \frac1\pi\int_0^{2\pi}dt\int_{1/2}^1\frac{r\,dr}{(1-r)^2}
\left(\int_r^1|B'(\rho e^{it})|\,d\rho\right)^2\leq\\&\leq
\frac4\pi\int_0^{2\pi}dt\int_{1/2}^1|B'(\rho{}e^{it})|^2\,d\rho\leq
8\int_{\D}|B'(z)|^2\,dm_2(z)=8n.\qedhere
\end{aligned}
$$

So, in order to complete the proof of Proposition~\ref{p1} it remains
to observe that $2\sqrt{8n+1}\leq 6\sqrt{n}$ for any positive integer
$n$.
\end{proof}

\bigskip
\section{The class $\RatU_n$ and the proof of Theorem~\ref{t0}}

Before proving Theorem~\ref{t0} let us give some simple criteria that
imply that a given rational function belongs to the class $\RatU_n$.

\subsection{Univalent functions in the space $\RatU_n$}
Take a function $R\in\Rat_n$ having all its poles $b_1,\dots,b_m$,
$m\leq n$, outside $\ov\D$. For $j=1,\dots,m$ we put $a_j=1/\ov{b}_j$,
and for $k=1,\dots,m$ we define the corresponding (finite) Blaschke
products
$$
B_k(z)=\prod_{j=1}^k\frac{\ov{a}_j}{|a_j|}\cdot
\frac{z-a_j}{\ov{a}_jz-1}.
$$
Let, also, $B_0(z)\equiv 1$ and $B=B_m$. Since $R=a+g$, where $a\in\C$
and $g\in K_B$ (the model space $K_B$ is defined in the proof of
Proposition~\ref{p1} above), then there exists a set of complex
coefficients $\{c_1,\dots,c_m\}$ such that
\begin{equation}\label{xxx}
R=a+\sum_{k=1}^m \frac{c_k\sqrt{1-|a_k|^2}}{1-\ov{a}_kz}B_{k-1}(z)
\end{equation}
(we recall, that the system of function
$\{\sqrt{1-|a_k|^2}B_{k-1}(z)/(1-\ov{a}_kz)\}_{k=1}^m$ forms an orthonormal
basis in the space $K_B$, see, for instance, \cite[Chap.~V]{nik80b}).

Assume that $a_1\in(0,1/\sqrt{2})$ and $|a_j|\leq|a_k|$ for $j\leq k$.
As it was shown in the proof of Theorem~2 in \cite{fed06}, if the set
of coefficients $\{c_1,\dots,c_m\}$ is such that $c_1=1$ and
$$
\sum_{k=2}^m|c_k|\sqrt{\frac{1+|a_k|}{1-|a_k|}}
\sum_{j=1}^k\frac{1+|a_j|}{1-|a_j|}\leq
\frac{a_1\sqrt{1-a_1^2}(1-2a_1^2)}{(1+a_1)^4},
$$
then the function $R$ defined by \eqref{xxx} satisfies the condition
$\re{R'(z)}>0$ in $\D$ and hence it is univalent in $\D$.

Therefore, for any set of points
$\{b_1,\dots,b_m\}\subset\C\setminus\ov\D$ there exists a function from
$\RatU_n$ with poles exactly at these points. Further examples of
univalent rational functions were obtained, for instance, in
\cite{bf11}.


\subsection{Proof of Theorem~\ref{t0}: the upper bound}

The estimate \eqref{e-upp} in Proposition~\ref{pro-upp} is a
direct corollary of Proposition~\ref{p1}. Indeed, if $R\in\RatU_n$, then
$$
\int_{\D}|R'(z)|^2\,dm_2(z)=\pi\Area(R(\D))\leq\pi^2\|R\|_{\infty,\T}^2,
$$
where $\Area(\Om)$ is the area of the domain $\Om\subset\C$, and hence
$$
6\sqrt{n}\left(\int_{\D}|R'(z)|^2\,dm_2(z)\right)^{1/2}
\leq 6\pi\sqrt{n}\,\|R\|_{\infty,\T}.\qedhere
$$


\subsection{Integral means spectrum for bounded univalent functions}

Let us recall the definition of the integral means spectrum for bounded
univalent functions (for a detailed exposition see \cite[Chap.~8]{pom92b} and
\cite[Chap.~VIII]{gm06b}). For a function $f$ and numbers $t\in\R$ and $r$,
$0<r<1$, let
$$
M_t[f'](r)=\frac1{2\pi}\int_0^{2\pi}|f'(re^{i\theta})|^t\,d\theta,
$$
and let
$$
\be_f(t)=\limsup_{r\to1-}\frac{\log{M_t[f'](r)}}{|\log(1-r)|}.
$$
Thus, $\be_f(t)$ is the smallest number $\be$ such that for any $\ep>0$
$$
M_t[f'](r)=O\biggl(\frac1{(1-r)^{\be+\ep}}\biggr)
$$
as $r\to1-$. Furthermore, let $B_b(t)=\sup_f\be_f(t)$, where the
supremum is taken over all bounded univalent in $\D$ functions.

L.~Carleson and P.\,W.~Jones conjectured in \cite{cj92} that
$$
B_b(1)=\frac{1}{4}.
$$
It is worth to mention that in \cite{cj92} it was proved that
$$
B_b(1)=\sup_f\limsup_{n\to\infty}\frac{\log(n|a_n|)}{\log{n}},
$$
where the supremum is taken over all functions $f(z)=\sum_{n=0}^\infty
a_nz^n$ which are bounded and univalent in $\D$.

The conjecture of Carleson and Jones remains open. There are several known
estimates of the number $B_b(1)$ from above and from below, see
\cite[Chap.~8]{pom92b}, \cite[Chap.~VIII]{gm06b}, \cite{hsh05} and references
therein. The best known estimate of $B_b(1)$ from below, $B_b(1)>0.23$, was
obtained by D.~Beliaev and S.~Smirnov \cite{bs04} for the conformal mappings
onto domains bounded by some fractal closed curves, while
H.~Hedenmalm and S.~Shimorin showed in \cite{hsh05} that $B_b(1)\leq 0.46$.
The novel techniques introduced in \cite{hsh05}
have been further improved in \cite{hsh07}, while in \cite{sol06} they were applied
to give a slightly smaller bound than $0.46$. In \cite{bh08} the estimate of the
function $B_b(t)$ was improved near the point $t=2$, and it also gives slight
lowering of $B_b(1)$.


\subsection{Proof of the lower bound in Theorem~\ref{t0}: Runge scheme}
The idea of this proof is to use the standard Runge approximation
scheme for a bounded univalent function $f$ with
almost extremal growth of $M_1[f'](r)$ as $r\to 1-$.

We start with the following lemma. For $z_0\in\C$ and $\rho>0$ let
$D(z_0,\rho)$ denote the open disk $\{z:\colon|z-z_0|<\rho\}$; we write
$D(\rho)$ in place of $D(0,\rho)$.

\begin{lem}\label{l1}
Let $0<\be<B_b(1)$. Then there exists a sequence of positive numbers
$\de_k\to0$, $k\to \infty$, and a sequence of functions $f_k$ such that

\smallskip
{\rm (i)} $f_k$ is univalent in the disk $D(1+4\de_k)$ and
$\|f_k\|_{\infty,D(1+4\de_k)}\leq1$\textup;

\smallskip
{\rm (ii)} $\inf_k |f'_k(0)|>0$\textup;

\smallskip
{\rm (iii)} $\int_0^{2\pi}|f_k'(e^{i\te})|\,d\te\geq c\de_k^{-\be}$ for
some absolute constant $c>0$.
\end{lem}

\begin{proof}
By the definition of $B_b(1)$ there exists a function $f$ which is
bounded and univalent in $\D$ and such that for some sequence
$r_k>1/2$, $r_k\to1-$, $k\to\infty$, we have
$$
\int_0^{2\pi}|f'(r_k e^{i\te})|\,d\te\geq (1-r_k)^{-\be}.
$$
Now put $f_k(z)=f(r_k z)$ and define $\de_k$ by $1+4\de_k=r_k^{-1}$.
Obviously, the functions $f_k$ have the properties (i)--(iii).
\end{proof}

Now we pass to the proof of the lower bound in Theorem~\ref{t0}. Let
$f_k$ and $\de_k\in(0,1/2)$ be as in Lemma \ref{l1}. Fix some
sufficiently large $k$. Take the circle $T_k=\{z\colon |z|=1+2\de_k\}$
and split it into the union of $N_k$ equal arcs $I_j$, $j=1,\dots,N_k$,
and for each $j$ take a point $\ze_j\in I_j$. The number
$N_k=N_k(\de_k)$ will be chosen later. For $|z|<1+2\de_k$, we have
$$
f_k(z)=\frac1{2\pi i}\int_{T_k}\frac{f_k(\ze)\,d\ze}{\ze-z}.
$$

Now, fix a positive integer $m$ and define the rational function $R$ of
degree at most $m N_k$ by the formula
$$
R(z)=\frac1{2\pi i}\sum_{j=1}^{N_k}\sum_{l=0}^m
\frac{1}{(\ze_j-z)^{l+1}}\int_{I_j}(\ze_j-\ze)^l f_k(\ze)\,d\ze.
$$
Note that, though we omit the index, $R$ also depends on $k$ and $m$.
In what follows $A_1,A_2,\dots$ will denote positive constants, whose
values do not depend on $\de_k$ or $N_k$ (though they may depend on $m$
and $\inf_k |f'_k(0)|$ from Lemma \ref{l1}, (ii)).

It follows from the identity
$$
\frac{1}{\ze-z}=\frac{1}{(\ze_j-z)\bigg(1-\dfrac{\ze_j-\ze}{\ze_j-z}\bigg)}=
\sum_{l=0}^m\frac{(\ze_j-\ze)^l}{(\ze_j-z)^{l+1}}+\frac{1}{\ze-z}
\bigg(\frac{\ze_j-\ze}{\ze_j-z}\bigg)^{m+1}
$$
that
\begin{equation}\label{bab3}
f_k(z)-R(z)=\frac1{2\pi i}
\sum_{j=1}^{N_k}\int_{I_j}\frac{f_k(\ze)}{\ze-z}
\bigg(\frac{\ze_j-\ze}{\ze_j-z}\bigg)^{m+1}d\ze.
\end{equation}
Since $|\ze-\ze_j|\leq4\pi/{N_k}$ when $\ze\in I_j$ and
$|z-\ze|\geq\de_k$ when $\ze\in T_k$ and $|z|\leq1+\de_k$, we conclude
that
\begin{equation}\label{un}
|f_k(z)-R(z)|\leq\frac{A_1}{N_k^{m+1}\de_k^{m+2}}, \qquad |z|\leq
1+\de_k.
\end{equation}
Also, differentiating \eqref{bab3}
we obtain that
\begin{equation}\label{un12}
|f_k'(z)-R'(z)|\leq\frac{A_2}{N_k^{m+1}\de_k^{m+3}}, \qquad |z|\leq
1+\de_k.
\end{equation}

Let us verify now that for sufficiently large $N_k$ the function $R$ is
univalent in $\D$. We start with verification of local univalence of
$R$. Since the function $f_k$ is univalent in $D(1+4\de_k)$, we have
for $|z|\leq 1+\de_k$ by the classical distortion theorem (see, e.g.,
\cite[Chap.~1]{pom92b})
$$
\big|f_k'(z)\big|\geq A_3|f'_k(0)|\de_k\geq A_4\de_k.
$$
Now, fix some $\ep>0$ and take
$$
N_k=\Big[\de_k^{-\frac{m+4}{m+1}-\ep}\Big]+1,
$$
where $[x]$ stands for the integer part of $x$. Then
$N_k^{-m-1}\de_k^{-m-3}=o(\de_k)$ when $k\to\infty$, and so for
sufficiently large $k$ we have
$$
\big|R'(z)\big|\geq A_4\de_k/2
$$
when $|z|\leq 1+\de_k$. Hence, $R$ is locally univalent in
$D(1+\de_k)$.

As the next step we need to check the injectivity of $R$ in $\D$.
Assume that there exist two distinct points $z_1,z_2\in\D$ such that
$R(z_1)=R(z_2)$. We consider two different cases:

\begin{enumerate}
\item[(a)]
$|z_1-z_2|\geq\de_k/2$, that is, the disks $D(z_1,\de_k/4)$
and $D(z_2,\de_k/4)$ are disjoint;
\item[(b)]
$|z_1-z_2|<\de_k/2$.
\end{enumerate}

In case (a) we show that the equation $f_k(z)-w$, where $w=R(z_1)$, has
roots in both disks $D(z_1,\de_k/4)$ and $D(z_2,\de_k/4)$ which
contradicts the univalence of $f_k$. Take $z\in D(1+\de_k)$. We have
\begin{multline*}
\big|R(z)-R(z_1)\big|\geq
\big|f_k(z)-f_k(z_1)\big|-\big|f_k(z)-R(z)\big|-\big|f_k(z_1)-R(z_1)\big|\geq\\
\geq \big|f_k(z)-f_k(z_1)\big|-\frac{2A_1}{N_k^{m+1}\de_k^{m+2}}.
\end{multline*}
It follows from the Koebe theorem (see \cite[Chap.~1]{pom92b}) and from
the property (ii) of Lemma \ref{l1} that for $z$ with $|z-z_1|=\de_k/4$
we have
$$
\big|f_k(z)-f_k(z_1)\big|\geq\frac14|z-z_1|\cdot|f_k'(z_1)|\geq A_5
\de_k^2.
$$
By the choice of $N_k$ we have $4A_1N_k^{-m-1}\de_k^{-m-2}<A_5\de_k^2$
for sufficiently large $k$, and so
$$
\big|R(z)-R(z_1)\big|\geq A_6\de_k^2
$$
as $|z-z_1|=\de_k/4$. At the same time, by \eqref{un}, we have
$|f_k(z)-R(z)|\leq A_7\de_k^{2+(m+1)\ep}$. As a standard application of
the Rouch\'e theorem we conclude that the equation $f_k(z)-w$ has a
root in the disk $D(z_1,\de_k/4)$ if $k$ is sufficiently large. But the
same arguments show that this equation also has a root in the disk
$D(z_2,\de_k/4)$, which is clearly impossible since $f_k$ is univalent.

In the case (b) an application of the Rouch\'e theorem together with
above estimates shows that the function $f_k-w$ has two zeros in the
disk $D(z_1,\de_k)$, again a contradiction. Thus, $R$ is univalent in
$\D$ for sufficiently large $k$. Also it is clear from \eqref{un} that
the norms $\|R\|_{\infty,\D}$ are bounded uniformly with respect to $k$.

To complete the proof, notice that, by \eqref{un12},
$$
\int_{\T}\big|f_k'(z)-R'(z)\big|\,dm(z)\leq
\frac{A_2}{N_k^{m+1}\de_k^{m+3}}\leq A_8\de_k^{1+(m+1)\ep}.
$$
Then, using Lemma~\ref{l1}, (iii), we finally obtain that for any $\ep>0$
$$
\int_{\T}|R'(z)|\,dm(z)\geq A_9\de_k^{-\be}\geq A_{10}
(mN_k)^{\be\frac{m+1}{m+4 +(m+1)\ep}}
$$
for sufficiently large $k$. Recall that the degree of $R$ is at most $mN_k$.
We conclude that
$$
\ga_0\geq\be\frac{m+1}{m+4 +(m+1)\ep}
$$
for any $\ep>0$ and any $m\in\mathbb{N}$. Hence, $\ga_0\geq\be$. Since
$\be$ in Lemma~\ref{l1} was an arbitrary number less than $B_b(1)$ we
finally get $\ga_0\geq B_b(1)$. \qed


\subsection{Proof of the lower bound in Theorem~\ref{t0}: Kayumov's approach}

In this subsection we present another proof of Theorem~\ref{t0}. Take
some $\ep>0$ and a function $f\in H^\infty$ such that
$\be_f(1)>B_b(1)-\ep$. Without loss of generality we may assume that
$f(0)=0$ and $f'(0)=1$, so that the function $f$ belongs to the class
$S$.

From now on we will follow the idea proposed by I.~Kayumov in
\cite{kay08}. Let us choose $r_k\to 1$ such that
$$
\be_f(1)=\lim_{k\to\infty}\frac{\log{M_1[f'](r_k)}}{|\log(1-r_k)|},
$$
and such that
$$
r_k=1-\dfrac{5\log{n_k}}{n_k}
$$
for some integers $n_k$.
Assume that $f(z)=\sum_{j=1}^\infty a_jz^j$ is the Taylor expansion of $f$ at
the origin and consider the sequence of polynomials $(P_k)_{k=1}^\infty$
such that
$$
P_k(z):=\sum_{j=1}^{n_k}a_jr_k^jz^j.
$$
In \cite{kay08} it was proved that these polynomials are univalent in $\D$.
For the polynomials $P_k$ we have
$$
\limsup_{k\to\infty}\dfrac{\log\ell(P_k)}{\log{n_k}}\geq \be_f(1)>B_b(1)-\ep.
$$

Note that the norms $\|P_k\|_{\infty,\T}$ of the polynomials $P_k$
are not necessarily bounded. However, by a classical result of E. Landau,
the norm of the projector
$$
\varPi_m\colon\sum_{j=0}^\infty b_j z^j\to\sum_{j=0}^m b_jz^j
$$
from $H^\infty$ to the space of all analytic polynomials of degree at
most $m$ equipped with $L^\infty$-norm satisfies
$\|\varPi_m\|\sim\frac{1}{\pi}\log{m}$ as $m\to\infty$ (see
\cite[Section 8.5]{dur} for details). Thus,
$$
\|P_k\|_{\infty,\T}\leq C\log{n_k}
$$
for some absolute positive constant $C$. Let us define the polynomials
$Q_k(z):=(\log n_k)^{-1}P_k(z)$. We have
$$
\begin{aligned}
\ga_0\geq \limsup_{k\to\infty}\dfrac{\log\ell(Q_k)}{\log{n_k}} & =
\limsup_{k\to\infty}\dfrac{\log\ell(P_k)}{\log{n_k}}-
\lim_{k\to\infty}\dfrac{\log \log n_k}{\log{n_k}} =\\
&=\limsup_{k\to\infty}\dfrac{\log\ell(P_k)}{\log{n_k}}>B_b(1)-\ep,
\end{aligned}
$$
whence $\ga_0\geq B_b(1)$.
\qed

\bigskip
\section{Theorem~\ref{t0}, Nevanlinna domains and related topics}

As it was mentioned in Introduction, Problem~\ref{pr1} and our main
result (Theorem~\ref{t0}) are related with the concept of a Nevanlinna
domains and therefore, with the problem of uniform approximability of
functions by polyanalytic polynomials on compact subsets of the complex
plane.

Take an integer $m\geq1$. Recall that a function $f$ is said to be
polyanalytic of order $m$ (or, shortly, $m$-analytic) on an open set
$U\subset\C$ if it is of the form
\begin{equation}\label{paf}
f(z)=\ov{z}^{m-1}f_{m-1}(z)+\cdots+\ov{z}f_1(z)+f_0(z),
\end{equation}
where $f_0,\dots,f_{m-1}$ are holomorphic functions in $U$; one denotes
by $\sHol_m(U)$ the class of all $m$-analytic functions on $U$. It is
clear, that any function $f\in\sHol_m(U)$ satisfies in $U$ the elliptic
partial differential equation $\ov\d{}^mf=0$, where $\ov\d$ is the
standard Cauchy--Riemann operator in $\C$, and therefore, the uniform
convergence preserves the polyanalyticity property. Furthermore, by
polyanalytic polynomials we mean polyanalytic functions such that all
functions $f_0,\dots,f_{m-1}$ from the representation \eqref{paf} are
polynomials in the complex variable.

For a compact set $X\subset\C$ let us denote by $C(X)$ the space of all
continuous complex valued functions on $X$ endowed with the standard
uniform norm $\|f\|_X=\max_{z\in X}|f(z)|$, $f\in C(X)$. Moreover, let
$A_m(X)=C(X)\cap\sHol_m(X^\circ)$, where $X^\circ$ stands for the
interior of $X$, and $P_m(X)=\{f\in C(X)\colon \forall\ep>0$ there
exists $m$-analytic polynomial $P$ such that $\|f-P\|_X<\ep\}$. It is
clear, that $P_m(X)\subset A_m(X)$. Now we are able to state the
approximation problem mentioned above.

\begin{prb}\label{appr}
Let $m\geq2$. What conditions on $X$ are necessary and sufficient in
order that
\begin{equation}\label{app}
A_m(X)=P_m(X)?
\end{equation}
\end{prb}

We restrict ourselves to the case $m\geq2$ and exclude the case $m=1$
from the consideration since, by the remarkable Mergelyan theorem,
$A_1(X)=P_1(X)$ if and only if the set $\C\setminus X$ is connected.

The study of Problem~\ref{appr} started in the middle of 1970s as the
study of the problem about uniform approximation by polyanalytic
rational functions (i.e., by functions of the form \eqref{paf} such
that all functions $f_0,\dots,f_{m-1}$ are rational functions) having
their poles outside $X$. Several necessary and sufficient conditions in
the latter problem were obtained in \cite{ofa75}, \cite{tw81} and
\cite{car85}. For instance, the following result is a direct
consequence of Theorem~2 in \cite{car85} and the well-known Runge
method of `pole-pushing': if the set $\C\setminus X$ is connected, then
$A_m(X)=P_m(X)$ for any $m\geq2$. However, as was shown
in~\cite{fed96}, even in the simplest case when the compact set~$X$ has
disconnected complement (namely, when $X$~is a~rectifiable contour
in~$\C$), the solution of Problem~\ref{appr} is formulated in terms of
special analytic properties of~$X$. In particular, if $X$~is an
arbitrary circle in~$\C$, then the equality \eqref{app} fails, but it
is satisfied (for any integer $m\geq2$) for any closed polygonal Jordan
curve as well as for any non-degenerate ellipse (which is not a~circle)
$X$ in~$\C$. In 1990s several results on approximability of functions
by polyanalytic polynomials have been obtained using the concept of a
Nevanlinna domain which was introduced in \cite{fed96} and
\cite{cfp02}. For instance, Theorem~2.2 in \cite{cfp02} says that
$A_m(X)=P_m(X)$ for a Carath\'eodory compact set $X$ if and only if any
bounded connected component of the set $\C\setminus X$ is not a
Nevanlinna domain. We recall that $X$ is a Carath\'eodory comapct set
if $\d{X}=\d\widehat{X}$, where $\widehat{X}$ is the union of $X$ and
all bounded connected components of $\C\setminus X$. An interested
reader can find a comprehensive survey about Problem~\ref{appr} and
certain related problems in \cite{mpf13}.

\smallskip
Let us now formulate the definition of a Nevanlinna domain (see
\cite[Definition~3]{fed96} and \cite[Definition~2.1]{cfp02}). For an
open set $U\subset\C$ let $H^\infty(U)$ be the space of all bounded
holomorphic functions in $U$. In view of the classical Fatou theorem,
every function $f\in H^\infty(\D)$ has finite angular boundary values
$f(\ze)$ for almost all (with respect to the Lebesgue measure on $\T$)
$\ze\in\T$.

\begin{dfn}
A bounded simply connected domain $\Om$ is called a Nevanlinna domain
if there exists two functions $u,v\in H^\infty(\Om)$, $v\noteq0$, such
that the equality
$$
\ov{z}=\frac{u(z)}{v(z)}
$$
holds almost everywhere on $\d\Om$ in the sense of conformal mapping,
which means that the equality of angular boundary values
$$
\ov{\varphi(\ze)}=\frac{(u\circ\varphi)(\ze)}{(v\circ\varphi)(\ze)}
$$
holds for almost all points $\ze\in\T$, where $\varphi$ is some
conformal mapping from $\D$ onto $\Om$.
\end{dfn}

Let $\ND$ be the class of all Nevanlinna domains. Note that the
definition of Nevanlinna domain is consistent since it does not depend
on the choice of $\varphi$. Furthermore, in view of Luzin--Privalov
boundary uniqueness theorem, the quotient $u/v$ is uniquely defined in
(a Nevanlinna domain) $\Om$.

It can be easily verified that $\D\in\ND$ but any domain bounded by any
closed polygonal Jordan curve as well as by any non-degenerate ellipse
is not in $\ND$.

As it was proved in \cite{cfp02}, the property $\Om\in\ND$ is equivalent
to the following property of conformal mapping $\varphi$ from $\D$ onto
$\Om$: there exist two functions $u_1,v_1\in
H^\infty(\ov\C\setminus\ov\D)$ with $v_1\noteq0$ such that the equality
of angular boundary values
$$
\varphi(\ze)=\frac{u_1(\ze)}{v_1(\ze)}
$$
holds for almost all points $\ze\in\T$, where $u_1(\ze)$ and $v_1(\ze)$
are angular boundary values of $u_1,v_1$, evaluated from the domain
$\ov\C\setminus\ov\D$. In this case one says that $\varphi$ admits a
pseudocontinuation (or pseudocontinuation of Nevanlinna type, because
the quotient $u_1/v_1$ belongs to the Nevanlinna class in
$\ov\C\setminus\ov\D$).

This property implies the following useful description of Nevanlinna
domains (see Theorem~1 in \cite{fed01}): $\Om\in\ND$ if and only if
$\varphi\in K_\varTheta:=H^2\ominus(\varTheta H^2)$ for some inner
function $\varTheta$ (i.e., $\varTheta\in H^\infty(\D)$ and
$|\varTheta(\ze)|=1$ for a.e. $\ze\in\T$).

Therefore, if $R$ is a rational functions having their poles outside
$\ov\D$ and if $R$ is univalent in $\D$, then $R(\D)\in\ND$. Of course,
the boundary of the domain $R(\D)$ is analytic.

Several interesting and important problems arise in connection with the
concept of a Nevanlinna domain. One of them is the problem of
description of possible regularity (or irregularity) of boundaries of
Nevanlinna domains. The following question was posed in \cite{fed01}
and remains open:

\begin{prb}\label{pr2}
Do there exist Nevanlinna domains with unrectifiable boundaries?
\end{prb}

Notice that several results about regularity of boundaries of
Nevanlinna domains were recently obtained in \cite{fed06} and
\cite{bf11}. It was proved that there exist Nevanlinna domains with
$C^1$, but not $C^{1,\al}$, $\al\in(0,1)$, boundaries as well as that
Nevanlinna domains may have `almost' unrectifiable boundaries. The
latter means that there exists such bounded univalent in $\D$ function
$f$ that $f$ admits a pseudocontinuation and $f'\notin H^p$ for any
$p>1$ (as usual, $H^p$ stands for the Hardy space in $\D$).

Let us show how Theorem~\ref{t0} suggests that the answer to the
question stated in Problem~\ref{pr2} is positive. By a contour we will
mean the boundary of some Jordan domain (not necessarily rectifiable)
and by Nevanlinna contours we mean boundaries of Jordan Nevanlinna
domains.

It is not difficult to prove that analytic Nevanlinna contours are
dense (in the sense of Hausdorff metric) in the set of all contours in
$\C$, so that in any neighborhood of an arbitrary contour in $\C$ there
exists an analytic Nevanlinna contour. Indeed, let $G$ be a Jordan
domain in $\C$ and let $f$ be some conformal mapping from $\D$ onto
$G$. Approximating $f$ by appropriate univalent in $\D$ polynomial
(which is clearly possible) we obtain (in view of the aforesaid) some
analytic Nevanlinna contour lying in a given $\ep$-neighborhood of
$\d{G}$. It is clear that the degrees $n$ of the corresponding
polynomials should grow to $\infty$ whenever $\ep\to 0$. In view of
Theorem~\ref{t0} the lengths of boundaries of corresponding analytic
Nevanlinna domains may grow at least as $n^\ga$ when $n\to\infty$. This
observation supports the conjecture that Nevanlinna domains with
unrectifiable boundaries do exist.

The observation that the lengths of the boundaries of rational images
of the unit disk may grow by a power low with the degree of the
corresponding mapping function has one more interesting interpretation
related to the concept of a quadrature domain.

Let us briefly recall this notion. For a bounded domain $\Om$ let
$A^2(\Om)$ be the standard Bergman space in $\Om$ (it means that
$A^2(\Om)$ consists of all holomorphic functions $f$ in $\Om$ such that
$|f|^2$ is integrable with respect to the planar Lebesgue measure in
$\Om$). The following definition may be found in many sources (see, for
instance, \cite{gsh05} and references therein).

\begin{dfn}
A bounded domain $\Om$ is called a \textup(classical\textup) quadrature
domain if there exist a finite set of points
$\{z_1,\dots,z_k\}\subset\Om$ and a set of complex numbers
$\{a_{js}\colon j=1,\dots,k; s=1,\dots,n_j\}$ {\rm(}where
$k,n_1,\dots,n_k$ are some positive integers and
$a_{jn_{j-1}}\ne0${\rm)} such that the equality
\begin{equation}\label{qi}
\int_{\Om}f(z)\,dxdy=\sum_{j=1}^k\sum_{s=0}^{n_j-1}a_{js}f^{(s)}(z_j)
\end{equation}
is satisfied for every function $f\in A^2(\Om)$.
\end{dfn}

Equality \eqref{qi} is traditionally called a quadrature identity
and the number $n=\sum_{j=1}^kn_j$ is the order of this quadrature
identity. It is also appropriate to refer to the points
$\{z_1,\dots,z_k\}$ as to the nodes of the quadrature identity \eqref{qi}.

Let now $R\in\RatU_n$, $R(0)=0$, and let $a_1,\dots,a_k \in \mathbb{C}
\setminus \overline{\mathbb{D}}$ be all poles of $R$ of multiplicities
$n_1,\dots,n_k$. As it was shown in \cite[Chap.~14]{dav74b}, the domain
$R(\D)$ is in this case a quadrature domain and the set of nodes of the
corresponding quadrature identity is $1/\ov{a}_1,\dots,1/\ov{a}_k$.
Moreover, any quadrature domain with analytic boundary is of this form.

In view of Theorem~\ref{t0} the lengths of the boundaries of quadrature
domains may grow by a power law $k^\ga$ with the order $k$ of the
corresponding quadrature identity.

\end{document}